\documentclass[12pt, twoside, leqno]{article}

\input xy

\xyoption{all}

\usepackage{times,calrsfs}

\usepackage{amsmath,amsthm}
\usepackage{amssymb}

\usepackage{enumerate}

\usepackage{graphicx}

\newtheorem{thm}{Theorem}
\newtheorem{prop}[thm]{Proposition}

\newtheorem{cor}[thm]{Corollary}

\newtheorem{ques}[thm]{Question}

\theoremstyle{definition}

\newtheorem*{xrem}{Remark}

\numberwithin{equation}{section}

\renewcommand{\bf}[1]{\mathbf{#1}}
\renewcommand{\rm}[1]{\mathrm{#1}}
\renewcommand{\cal}[1]{\mathcal{#1}}

\newcommand{\bbN}{\mathbb{N}}

\newcommand{\bbZ}{\mathbb{Z}}

\newcommand{\sfE}{\mathsf{E}}
\renewcommand{\d}{\mathrm{d}}

\newcommand{\rmh}{\mathrm{h}}

\newcommand{\F}{\mathcal{F}}

\renewcommand{\L}{\Lambda}

\renewcommand{\S}{\Sigma}

\newcommand{\eps}{\varepsilon}

\newcommand{\s}{\sigma}
\renewcommand{\phi}{\varphi}

\renewcommand{\hat}[1]{\widehat{#1}}

\renewcommand{\qed}{\nolinebreak\hspace{\stretch{1}}$\Box$\vspace{7pt}}
\newcommand{\fin}{\nolinebreak\hspace{\stretch{1}}$\lhd$}

\renewcommand{\t}[1]{\tilde{#1}}
\renewcommand{\to}{\longrightarrow}

\newcommand{\op}{\mathrm{op}}

\begin{document}

%%%%% To ease editing, for IMPAN journals add:

\baselineskip=17pt

%%%%%%%%%%%%%%%%

\title{Entropy of probability kernels from the backwards tail boundary}

\author{Tim Austin
}

\date{}

\maketitle

%% Classification and key words; note that the 2010 classification is used:

\renewcommand{\thefootnote}{}



\renewcommand{\thefootnote}{\arabic{footnote}}
\setcounter{footnote}{0}

%%%%%%%%

\begin{abstract}
A number of recent works have sought to generalize the Kolmogorov-Sinai entropy of probability-preserving transformations to the setting of Markov operators acting on the integrable functions on a probability space $(X,\mu)$.  These works have culminated in a proof by Downarowicz and Frej that various competing definitions all coincide, and that the resulting quantity is uniquely characterized by certain abstract  properties.

On the other hand, Makarov has shown that this `operator entropy' is always dominated by the Kolmogorov-Sinai entropy of a certain classical system that may be constructed from a Markov operator, and that these numbers coincide under certain extra assumptions.  This note proves that equality in all cases.
\end{abstract}

Let $(X,\mu)$ be a standard Borel probability space, and let $P:X \to \Pr X$ be a probability kernel which preserves $\mu$. When it is needed, $\S_X$ will denote the $\s$-algebra of $X$. The triple $(X,\mu,P)$ is a \textbf{random probability-preserving} (`\textbf{p.-p.}') \textbf{system}.  Such a $P$ may be identified with a Markov operator $L^p(\mu)\to L^p(\mu)$ for any $p \in [1,\infty]$ (that is, an operator fixing $1_X$ and preserving both non-negativity and the integral), and the assumption that $X$ is standard Borel implies that any Markov operator arises this way~(\cite[Subsection 1.2]{Kai92}).

If $T:X \to X$ is a $\mu$-preserving measurable transformation, then one may define a probability kernel $U_T$ by setting $U_T(x,\,\cdot\,) := \delta_{Tx}$.  As a Markov operator this is simply the Koopman operator of $T$. In this way classical p.-p. systems give examples of random p.-p. systems.  These classical examples will sometimes be distinguished by calling them \textbf{non-random}.

Several recent works have sought to generalize the Kolmogorov-Sinai entropy of non-random p.-p. systems to the setting of random p.-p. systems.  This effort began with developments in quantum dynamical systems~\cite{AliAndFanTuy96}, and continued with several proposals for the `operator entropy' of probability kernels (equivalently, Markov operators)~\cite{GhyLanWal86,KamdeSam00,Mak00}.  In~\cite{DowFre05}, it was shown that these quantities all coincide, by showing that they have in common a list of properties which determine the relevant function uniquely.  These developments are summarized in~\cite[Chapter 11]{Dow11}.

For Makarov's definition of operator entropy (and hence also all the others), he showed in~\cite{Mak00} that it is always dominated by the Kolmogorov-Sinai entropy of a naturally-associated deterministic system: the backwards tail boundary of the associated shift-invariant measure on path space.  This note will show that these numbers are actually always equal.

This will need only some basic properties of operator entropy.  The following can all be found, for instance, in~\cite[Chapter 11]{Dow11}.  Firstly, one has
\[\rmh_{\rm{op}}(\mu,P) = \sup_{\F}\rmh_{\rm{op}}(\F,\mu,P),\]
where $\F$ runs over finite families of measurable functions $X \to [0,1]$, and where $\rmh_{\rm{op}}(\F,\mu,P)$ is a function defined on such data.  In addition:
\begin{itemize}
\item[(P1:] consistency under factors) if $\pi:(X_2,\mu_2,P_2)\to (X_1,\mu_1,P_1)$ is a factor map of random p.-p. systems and $\F$ is a finite family of measurable functions $X_1 \to [0,1]$, then
\[\rmh_{\rm{op}}(\{f\circ\pi\,|\ f\in\F\},\mu_2,P_2) = \rmh_{\rm{op}}(\F,\mu_1,P_1);\]
\item[(P2:] consistency with KS entropy) for a non-random p.-p. system $(X,\mu,T)$, one has
\[\rmh_{\rm{op}}(\mu,U_T) = \rmh_{\rm{KS}}(\mu,T),\]
where $\rmh_{\rm{KS}}$ denotes the Kolmogorov-Sinai entropy;
\item[(P3:] continuity in $L^1$) for every $k \in \bbN$ and $\eps > 0$ there is a $\delta > 0$ for which the following holds: if $\F = \{f_1,\ldots,f_k\}$ and $\cal{G} = \{g_1,\ldots,g_k\}$ are two finite families of measurable functions $X \to [0,1]$, then
\[\sum_{i=1}^k\|f_i - g_i\|_1 < \delta \quad \Longrightarrow \quad |\rmh_{\rm{op}}(\F,\mu,P) - \rmh_{\rm{op}}(\cal{G},\mu,P)| < \eps;\]
\item[(P4:] invariance under $P$) for any $(X,\mu,P)$ and $\F$ one has
\[\rmh_\op(\F,\mu,P) = \rmh_\op(P\F,\mu,P).\]
\end{itemize}

The backwards tail boundary is a very classical construction in the study of abstract Markov chains.  It is recalled as a $\sigma$-algebra in~\cite[Section 2]{Mak00}.  For the present paper, the wide-ranging survey~\cite{Kai92} offers a suitable basic reference (although beware that the entropy discussed in~\cite[Section 3]{Kai92} is quite unrelated to that studied here).  Similar material can also be found in some standard probability texts, such as in~\cite[Chapter IV]{Ros71}.

Given a random p.-p. system $(X,\mu,P)$, one first defines the corresponding shift-invariant measure $\t{\mu}$ on the path space $X^\bbZ$ by specifying its finite-dimensional marginals, thus:
\begin{multline}\label{eq:tmu-dfn}
\t{\mu}(X^{(-\infty;i)}\times A_i\times A_{i+1}\times \cdots \times A_{j-1}\times A_j\times X^{(j;\infty)})\\ = \int_{A_i}\int_{A_{i+1}}\cdots \int_{A_j}P(x_{j-1},\d x_j)P(x_{j-2},\d x_{j-1})\cdots P(x_i,\d x_{i+1})\,\mu(\d x_i).
\end{multline}
Giving $X^\bbZ$ the product $\s$-algebra $\S_X^{\otimes \bbZ}$, and letting $S:X^\bbZ\to X^\bbZ$ be the leftward coordinate-shift, this results in a non-random p.-p. system $(X^\bbZ,\t{\mu},S)$.  This probability space is called the \textbf{path space} associated to $(X,\mu,P)$. Abundant, simple examples show that the KS entropy of the shift on the path space need not equal $\rmh_\op(\mu,P)$.

Next, let $\t{\mu}^-$ be the marginal of $\t{\mu}$ on $X^{(-\infty;0]}$.  The rightward-shift, $S^{-1}$, descends to a well-defined transformation $R:X^{(-\infty;0]}\to X^{(-\infty;0]}$, but $R$ is no longer invertible.  Instead, its adjoint as an operator on $L^1(\t{\mu}^-)$ is given by the probability kernel
\[\t{P}((\ldots,x_{-1},x_0),\,\cdot\,) := P(x_0,\,\cdot\,).\]
The obvious coordinate projections now give the factor maps for a tower of random p.-p. systems:
\[(X^\bbZ,\t{\mu},S) \stackrel{\pi_{(-\infty;0]}}{\to} (X^{(-\infty;0]},\t{\mu}^-,\t{P}) \stackrel{\pi_0}{\to} (X,\mu,P).\]

Lastly, within $\S_X^{\otimes (-\infty;0]}$, consider the $\s$-subalgebras
\[\Phi_{\leq n} := \S^{\otimes (-\infty;n]}\otimes \{\emptyset,X\}^{\otimes (n;0]}\]
for each $n \in (-\infty;0]$. The reverse filtration $(\Phi_{\leq n})_{n\leq 0}$ is the \textbf{backwards filtration}, and $\Phi_{-\infty} := \bigcap_n \Phi_{\leq n}$ is the \textbf{backwards tail $\s$-algebra}.  Regarded as a $\s$-subalgebra of $\S_X^{\otimes \bbZ}$, the backwards tail is shift-invariant, so defines a factor $(X^\bbZ,\Phi_{-\infty},\t{\mu}|_{\Phi_{-\infty}},S)$.  Since $X^\bbZ$ is standard Borel, this factor may be generated up to negligible sets by an equivariant map, say
\[\phi: (X^\bbZ,\t{\mu},S) \to (\hat{X},\hat{\mu},\hat{S}),\]
whose target is another standard Borel system called the \textbf{backwards tail boundary} of $(X,\mu,P)$.  Since $\Phi_{-\infty}\leq \S_X^{\otimes (-\infty;0]}$, it follows that, up to a $\t{\mu}$-negligible set, this equivariant map factorizes through the coordinate projection $X^\bbZ\to X^{(-\infty;0]}$.  We have therefore produced a diagram of factor maps
\begin{center}
$\phantom{i}$\xymatrix{
& (X^\bbZ,\t{\mu},S) \ar^{\pi_{(-\infty;0]}}[d]\\
& (X^{(-\infty;0]},\t{\mu}^-,\t{P}) \ar[dl]\ar_{\pi_0}[dr]\\
(\hat{X},\hat{\mu},\hat{S}) && (X,\mu,P).
}
\end{center}

\begin{thm}\label{thm:1}
In the above situation, one has
\[\rm{h}_\op(\mu,P) = \rmh_\op(\t{\mu}^-,\t{P}) = \rmh_{\rm{KS}}(\hat{\mu},\hat{S}).\]
\end{thm}

\begin{proof}
The inequality $\rmh_\rm{op}(\mu,P) \leq \rmh_{\rm{KS}}(\hat{\mu},\hat{S})$ is~\cite[Theorem 2.8]{Mak00}. It follows from the convergence
\[(\t{P}^\ast)^n\t{P}^n \to \sfE_{\t{\mu}^-}(\,\cdot\,|\,\Phi_{-\infty}) \quad \hbox{as}\ n\to\infty\]
in the strong topology of operators on $L^2(\t{\mu}^-)$: see~\cite[Lemmas 2.6 and 2.7]{Mak00}.  Awareness of this convergence is actually at least as old as Rota's work~\cite{Rota62}.

Next, the inequality $\rmh_{\rm{KS}}(\hat{\mu},\hat{S}) \leq \rmh_\rm{op}(\t{\mu}^-,\t{P})$ is immediate, because
\begin{itemize}
\item by (P2), one has $\rmh_{\rm{KS}}(\hat{\mu},\hat{S}) = \rmh_\rm{op}(\hat{\mu},U_{\hat{S}})$, and
\item $(\hat{X},\hat{\mu},U_{\hat{S}})$ is a factor of $(X^{(-\infty;0]},\t{\mu}^-,\t{P})$, and $\rmh_{\rm{op}}$ is monotone under factor maps, since it is defined as a supremum over finite families of measurable functions, and for these we may apply (P1).
\end{itemize}

It only remains to show that $\rmh_\rm{op}(\t{\mu}^-,\t{P}) \leq \rmh_{\rm{op}}(\mu,P)$.  Makarov proves this in a special case in~\cite[Theorem 3.1]{Mak00}; we will now do so without his extra assumptions.

Let $\cal{G}$ be a finite set of measurable functions $X^{(-\infty;0]} \to [0,1]$.  It suffices to show that
\[\rmh_\op(\cal{G},\t{\mu}^-,\t{P}) \leq \rmh_\op(\mu,P),\]
since $\rmh_\op(\t{\mu}^-,\t{P})$ is then defined by supremizing over $\cal{G}$ on the left-hand side.

By (P3), it suffices to prove this for all finite $\cal{G}$ contained in some $\|\cdot\|_1$-dense subset of the space of measurable functions $X^{(-\infty;0]} \to [0,1]$.  We may therefore assume that there is some $m\geq 0$ such that every $g \in \cal{G}$ depends on only the coordinates $x_{-m+1},x_{-m+2},\ldots,x_0$ of $x \in X^{(-\infty;0]}$.

Having made this assumption, the Markov property of the law $\t{\mu}$ gives that, for a random string $(\ldots,x_{-1},x_0)$ drawn from $\t{\mu}^-$, the distributions of $(x_{-m+1},\ldots,x_0)$ and $(x_n)_{n \leq -m-1}$ are conditionally independent given $x_{-m}$.  This implies that for every $g \in \cal{G}$, the conditional expectation $\sfE(g\,|\,\Phi_{\leq -m})$ is of the form $Qg(x_{-m})$ for some $Qg:X \to [0,1]$, and hence $\t{P}^mg = \sfE(g\,|\,\Phi_{\leq -m})\circ S^m = Qg\circ\pi_0$.

Letting $\F := \{Qg\circ \pi_0\,|\ g \in \cal{G}\}$, an $m$-fold appeal to (P4) now gives
\[\rmh_\op(\cal{G},\t{\mu}^-,\t{P}) = \rmh_\op(\t{P}^m\cal{G},\t{\mu}^-,\t{P}) = \rmh_\op(\F,\t{\mu}^-,\t{P}).\]
Since all members of $\cal{F}$ are lifted from the random p.-p. system $(X,\mu,P)$, this last quantity is bounded above by $\rmh_\op(\mu,P)$, by (P1).
\end{proof}

\begin{xrem}
If $X$ is a finite set, then for any shift-invariant measure $\t{\mu}$ on $X^\bbZ$ the above $\s$-algebra $\Phi_{-\infty}$ defines the Pinsker factor of $(X^\bbZ,\t{\mu},S)$, which is the maximal factor of entropy zero.  However, in case $X$ is a general state space, this theory does not apply, and the factor $\Phi_{-\infty}$ may have any entropy in $[0,\infty]$. \fin
\end{xrem}

The following properties of $\rmh_{\rm{op}}$ are known, but may also be deduced quickly from Theorem~\ref{thm:1}:

\begin{itemize}
\item One always has $\rmh_\op(\mu,P) \leq \rmh_{\rm{KS}}(\t{\mu},S)$.  This was previously deduced in~\cite{Frej12}.  However, it also follows already from Makarov's inequality $\rmh_\op(\mu,P) \leq \rmh_{\rm{KS}}(\hat{\mu},\hat{S})$, since $(\hat{X},\hat{\mu},\hat{S})$ is a factor of $(X^\bbZ,\t{\mu},S)$.

\item If $(X_i,\mu_i,P_i)$ are two random p.-p. systems for $i=1,2$, and one defines $P_1\otimes P_2:X_1\times X_2\to \Pr(X_1\times X_2)$ by
\[(P_1\otimes P_2)((x_1,x_2),\,\cdot\,) := P_1(x_1,\,\cdot\,)\otimes P_2(x_2,\,\cdot\,),\]
then
\[\rmh_\op(\mu_1\otimes \mu_2,P_1\otimes P_2) = \rmh_\op(\mu_1,P_1) + \rmh_\op(\mu_2,P_2).\]
This was previously shown in~\cite{FrejFrej11}.  It now follows from the corresponding result for KS entropy, since the backwards tail boundary is functorial under products.
\end{itemize}

Theorem~\ref{thm:1} also suggests an obvious definition of conditional operator entropy on a factor, as requested in~\cite[Question 13.1.2]{Dow11}: if
\[(X_1,\mu_1,P_1) \to (X_2,\mu_2,P_2),\]
then the functoriality of the backwards tail boundary gives a factor map
\[(\hat{X}_1,\hat{\mu}_1,\hat{S}_1) \to (\hat{X}_2,\hat{\mu}_2,\hat{S}_2),\]
suggesting that the conditional entropy of the former extension should be the conditional Kolmogorov-Sinai entropy of the latter.  It might be interesting to find a formula for this conditional entropy which does not require the construction of the backwards tail boundaries.

Given a random p.-p. system $(X,\mu,P)$, one may also construct a non-random p.-p. system directly as a factor of it, without first ascending to $(X^{(-\infty;0]},\t{\mu}^-,\t{P})$.  There is a maximal such factor, constructed via the \textbf{deterministic $\s$-algebra} of $(X,\mu,P)$:
\[\Psi_\infty := \bigcap_{m\geq 1}\{A \in \S_X\,|\ P^{(m)}(x,A) \in \{0,1\}\ \hbox{for}\ \mu\hbox{-a.e.}\ x\}\]
(see~\cite{KreLin89}).  The map
\[\Psi_\infty\to \S_X:A\mapsto \{x\in X\,|\ P(x,A) = 1\}\]
takes values in $\Psi_\infty$, and in fact defines a $\mu$-preserving automorphism of that $\s$-subalgebra.  It therefore results from a factor map $\psi:X \to (X',\mu',S')$ to a non-random p.-p. system, and one may check easily that any member of $\S_X$ lifted from a non-random factor of $(X,\mu,P)$ must be a member of $\Psi_\infty$, so this construction gives the maximal non-random factor of $(X,\mu,P)$.

If one applies this construction to $(X^{(-\infty;0]},\t{\mu}^-,\t{P})$, then it simply gives the backwards tail $\s$-algebra again.  Since $(X,\mu,P)$ is itself a factor of $(X^{(-\infty;0]},\t{\mu}^-,\t{P})$, the maximal non-random factor of the former must be contained in the maximal non-random factor of the latter, and so we obtain a factor map of non-random systems
\begin{eqnarray}\label{eq:non-random-factor}
(\hat{X},\hat{\mu},\hat{S})\to (X',\mu',S').
\end{eqnarray}

It is worth observing that this factor map is sometimes not the identity, and that one must use $(\hat{\mu},\hat{S})$ rather than $(\mu',S')$ in Theorem~\ref{thm:1}. This can be seen in the following classical family of examples, which are essentially those in the closing pages of~\cite[Section IV.4]{Ros71} or in~\cite[Section 4]{Mak00}. Let $\bf{p} := (p_k)_{k \in \bbZ} \in [0,1]^\bbZ$, and let
\[\nu_\bf{p} := \bigotimes_{k\in \bbZ} ((1 - p_k)\delta_0 + p_k\delta_1) \in \Pr \bbZ_2^\bbZ,\]
where $\bbZ_2 := \bbZ/2\bbZ$.  Let $S:\bbZ_2^\bbZ\to \bbZ_2^\bbZ$ be the leftward shift, as previously, and define an associated probability kernel $P_\bf{p}:\bbZ_2^\bbZ\to \bbZ_2^\bbZ$ by
\[P_\bf{p}(x,\,\cdot\,) := \delta_{Sx}\ast \nu_\bf{p}.\]
Thus, as a doubly stochastic operator, $P_\bf{p}$ is the composition of $S$ with the convolution by $\nu_\bf{p}$.
Let $\mu := (\frac{1}{2}\delta_0 + \frac{1}{2}\delta_1)^{\otimes \bbZ}$, and $\mu^+ := (\frac{1}{2}\delta_0 + \frac{1}{2}\delta_1)^{\otimes [0;\infty)}$; clearly $\mu$ is $P_\bf{p}$-invariant for every $\bf{p}$.

The analyses in~\cite[Section IV.4]{Ros71} or~\cite[Section 4]{Mak00} give the following.

\begin{prop}\label{prop:ex}
If $\bf{p} \neq \boldsymbol{0}$, then these examples enjoy the following alternative:
\begin{enumerate}
\item[i)] if there is some $k_0$ such that $p_k = 0$ for all $k < k_0$, then one has a commutative diagram
\begin{center}
$\phantom{i}$\xymatrix{
(\hat{X},\hat{\mu},\hat{S}) \ar[d]\ar^-\cong[r] & (\{0,1\}^\bbZ,\mu,\rm{shift}) \ar^{\rm{coord. proj.}}[d]\\
(X',\mu',S') \ar_-\cong[r] & (\{0,1\}^{[0;\infty)},\mu^+,\rm{shift})
}
\end{center}
\item[ii)] if there are arbitrarily large $k$ for which $p_{-k} \neq 0$, but $\sum_{k < 0}p_k < \infty$, then
\[(\hat{X},\hat{\mu},\hat{S}) \cong (\{0,1\}^\bbZ,\mu,\rm{shift})\]
but
\[(X',\mu',S') \cong \rm{trivial};\]
\item[iii)] if $\sum_{k < 0}p_k = \infty$, then
\[(\hat{X},\hat{\mu},\hat{S}) \cong (X',\mu',S') \cong \rm{trivial}\]
\end{enumerate}
$\phantom{i}$\qed
\end{prop}

Proposition~\ref{prop:ex} can also be applied to the future tails of $(X,\mu,P)$ and $(X^\bbZ,\t{\mu},S)$, giving the analogous description of those tails depending now on $p_k$ for $k\to\infty$.  Since these may be chosen independently of $p_{-k}$ as $k\to\infty$, this shows that the backward and future tails need not be related.

A different example, also giving a backward tail boundary equal to a Bernoulli shift and a trivial forward tail boundary, is given in~\cite[Theorem 4.4]{Kai92} (although his `forward' and `backward' are the reverse of ours).

A modification of the preceding example shows that while $\rmh_\rm{op}$ behaves well under Cartesian products, it does not enjoy any obvious inequalities for general joinings.

\begin{cor}
There is a diagram of random p.-p. systems
\begin{center}
$\phantom{i}$\xymatrix{
& (X,\mu,P) \ar[dl]\ar[dr]\\
(X_1,\mu_1,P_1) && (X_2,\mu_2,P_2)
}
\end{center}
such that the two factors generate $(X,\mu,P)$, but $\rmh_\op(\mu_1,P_1) = \rmh_\op(\mu_2,P_2) = 0$ while $\rmh_\op(\mu,P) > 0$.
\end{cor}

\begin{proof}
Let $X := \bbZ_2^\bbZ\times \bbZ_2^\bbZ$ with the two obvious projections to $X_1 := X_2 := \bbZ_2^\bbZ$.  Let $\mu_1 = \mu_2$ be the Haar measure on $\bbZ_2^\bbZ$ and let $\mu := \mu_1\otimes \mu_2$.  Finally, let
\[\nu:= \Big(\frac{1}{2}\delta_{(0,0)} + \frac{1}{2}\delta_{(1,1)}\Big)^{\otimes \bbZ} \in \Pr X,\]
and let $P(x,\,\cdot\,) := \delta_{Sx}\ast \nu$.  Then the two coordinate projections $X \to X_i$ are both factor maps from $P$ to $P_\bf{p}$, the kernel defined previously, with $\bf{p} = (\ldots,1/2,1/2,\ldots)$, and by Proposition~\ref{prop:ex} this has trivial tail boundary and hence zero operator entropy.  However, the group homomorphism
\[X\to \bbZ_2^\bbZ:(x_1,x_2)\mapsto x_1 - x_2\]
is also a factor map from $P$, this time to the non-random leftward shift $S$ on $\bbZ_2^\bbZ$.  Therefore, by monotonicity, $\rmh_\op(\mu,P)$ is at least the KS entropy of the Bernoulli shift $S$ on $\bbZ_2^\bbZ$, which is $\log 2$.  (In fact, just a little more care shows that they are equal in this case.)
\end{proof}

This corollary suggests that there is no simple analog for $\rmh_\op$ of the notion of the Pinsker factor for a non-random p.-p. system, since the above example gives two factors of $(X,\mu,P)$ which both have zero entropy, but cannot be contained in a single factor of zero entropy.  (See~\cite[Question 13.1.4]{Dow11}.)

We finish by collecting some directions for further investigation.

\begin{itemize}
\item Firstly, one could easily generalize some of the definitions of $\rmh_{\rm{op}}$ to the setting of a $\mu$-preserving continuous-time semigroup $(P^t)_{t\geq 0}$ of Markov operators.  All of the arguments above should go through in that setting, using the standard analogous machinery for continuous-time Markov processes in the appropriate places: see, for instance,~\cite[Chapter 20]{Kal02}.

\item In~\cite{DowFre05}, Downarowicz and Frej also introduced a topological (as opposed to measure-preserving) version of operator entropy for a suitable class of Markov operators on compact metric spaces, and showed that it retains various classical properties of topological dynamical entropy.  It would be interesting to see whether it, too, could be reduced to an instance of that classical notion, perhaps using some kind of topological boundary.

\item In~\cite[Question 13.1.6]{Dow11}, Downarowicz asks how operator entropy behaves under convex combination of probability kernels.  It is not at all clear how the convex structure here interacts with backwards tail boundaries.  Mostly simply, given $(X,\mu)$ and two $\mu$-preserving transformations $S,T:X \to X$, it is not clear why the tail of $P(x,\,\cdot\,):= \frac{1}{2}(\delta_{Sx} + \delta_{Tx})$ should bear any relation to $S$ or $T$ themselves.  However, if $S$ and $T$ commute and are invertible, then one can say something: in that case a simple appeal to the norm ergodic theorem gives
\[P^nf = 2^{-n}\sum_{p,q\geq 0,\,p + q = n}\binom{n}{p}S^pT^qf \sim \sfE(f\,|\,\L)\circ T^n \quad \hbox{in}\ \|\cdot\|_1,\]
where $\L \leq \S_X$ is the $\s$-algebra of $S^{-1}T$-invariant sets.  This implies that the backward tail boundary of $P$ is just $(X,\L,\mu|_\L,S)$, which is a factor of both $(X,\mu,S)$ and $(X,\mu,T)$ and so has KS entropy bounded by either of their KS entropies.  If $S$ and $T$ commute but are not invertible, then we obtain this asymptotic behaviour upon ascending to the natural extension of the $\bbN^2$-action generated by $S$ and $T$, and this does not change the KS entropies of $S$ and $T$.

\item Finally, motivated by Proposition~\ref{prop:ex}, I think it might be interesting to answer the following:

\begin{ques}
Which extensions of non-random p.-p. systems can arise as the extension in~(\ref{eq:non-random-factor}) for some $(X,\mu,P)$?
\end{ques}

Using examples such as in Proposition~\ref{prop:ex} as a building-block, it is easy to obtain any relatively Bernoulli extension this way.  On the other hand, I suspect this extension is always relatively mixing, so there are some restrictions.
\end{itemize}

\normalsize
\baselineskip=17pt

\bibliographystyle{abbrv}
\bibliography{bibfile}

\noindent\small{Courant Institute of Mathematical Sciences}\\
\small{New York University}\\
\small{251 Mercer St, New York, NY 10012, U.S.A.}\\
\small{E-mail: tim@cims.nyu.edu}

\end{document}